\documentclass[11pt]{article}
\usepackage[utf8]{inputenc}

\usepackage{geometry} 
\usepackage{amssymb}
\usepackage{amsmath}
\usepackage{amsthm}
\usepackage{enumerate}
\usepackage{graphicx, wrapfig}
\usepackage{fullpage}
\usepackage{caption}
\usepackage{subcaption}
\usepackage{lipsum}
\usepackage{algorithm}
\usepackage{algorithmic}
\usepackage[hidelinks]{hyperref}
\usepackage{color}
\usepackage{thmtools}
\usepackage{comment}
\usepackage{stmaryrd}

\title{Bistability and oscillatory behaviours of cyclic feedback loops}

\author{Jules Guilberteau 
\thanks{ 
       Sorbonne Université, CNRS, Université Paris Cité, Inria, Laboratoire Jacques-Louis Lions (LJLL), F-75005 Paris, France. {\tt\small jules.guilberteau@sorbonne-universite.fr}}
}

\date{}
\newtheorem{theorem}{Theorem}
\newtheorem{prop}{Proposition}

\newtheorem{lem}{Lemma}
\newtheorem*{theorem*}{Theorem}

\theoremstyle{definition}
\newtheorem{definition}{Definition}

\theoremstyle{definition}
\newtheorem*{rem}{Remark}
\newtheorem*{ex}{Example}

\makeatletter
\def\thmhead@plain#1#2#3{%
  \thmname{#1}\thmnumber{\@ifnotempty{#1}{ }\@upn{#2}}%
  \thmnote{ {\the\thm@notefont#3}}}
\let\thmhead\thmhead@plain
\makeatother

\newcommand{\R}{\mathbb{R}}

\begin{document}
\maketitle

\begin{abstract}
In this paper, we study the stability of an Ordinary Differential Equation (ODE) usually referred to as Cyclic Feedback Loop, which typically models a biological network of $d$ molecules where each molecule regulates its successor in a cycle ($A_1\rightarrow A_2\rightarrow \ldots  \rightarrow A_{d-1} \rightarrow  A_d \rightarrow  A_1$). Regulations, which can be either positive or negative, are modelled by increasing or decreasing functions. We make a complete analysis of this model for a wide range of functions (including affine and Hill functions) by determining the parameters for which bistability and oscillatory behaviours arise. 
These results encompass previous theoretical studies of gene regulatory networks, which are particular cases of this model. 


\end{abstract}

{\bf Keywords :} Gene regulatory network, Repressilator, Toggle switch, Stability analysis, Multistability, Periodic orbit.

\section{Introduction}

We aim to characterise the stability of the ODE system
\begin{align}
\begin{cases}
\dot x_1=\alpha_1 f_1(x_d)-x_1\\
\dot x_2=\alpha_2 f_2(x_1)-x_2\\
\vdots\\
\dot x_d=\alpha_d f_d(x_{d-1})-x_d
\end{cases}  
\label{ODE dD}
\end{align}
where  $f_1,..., f_d \in \mathcal{C}^1(\R_+, \R_+)$ (with $\R_+=[0, +\infty)$) are non-negative functions, at least one of them is bounded, and $\alpha_1,..., \alpha_d$ are positive parameters.  Throughout this paper,  we use the convention $x_0=x_d$, which allows us to write \eqref{ODE dD} under the compacted form 

\[\quad \forall i\in \{1,..., d\}, \quad \dot x_i=\alpha_if_i(x_{i-1})-x_i.\]

\begin{figure}
\centering
\includegraphics[scale=0.2]{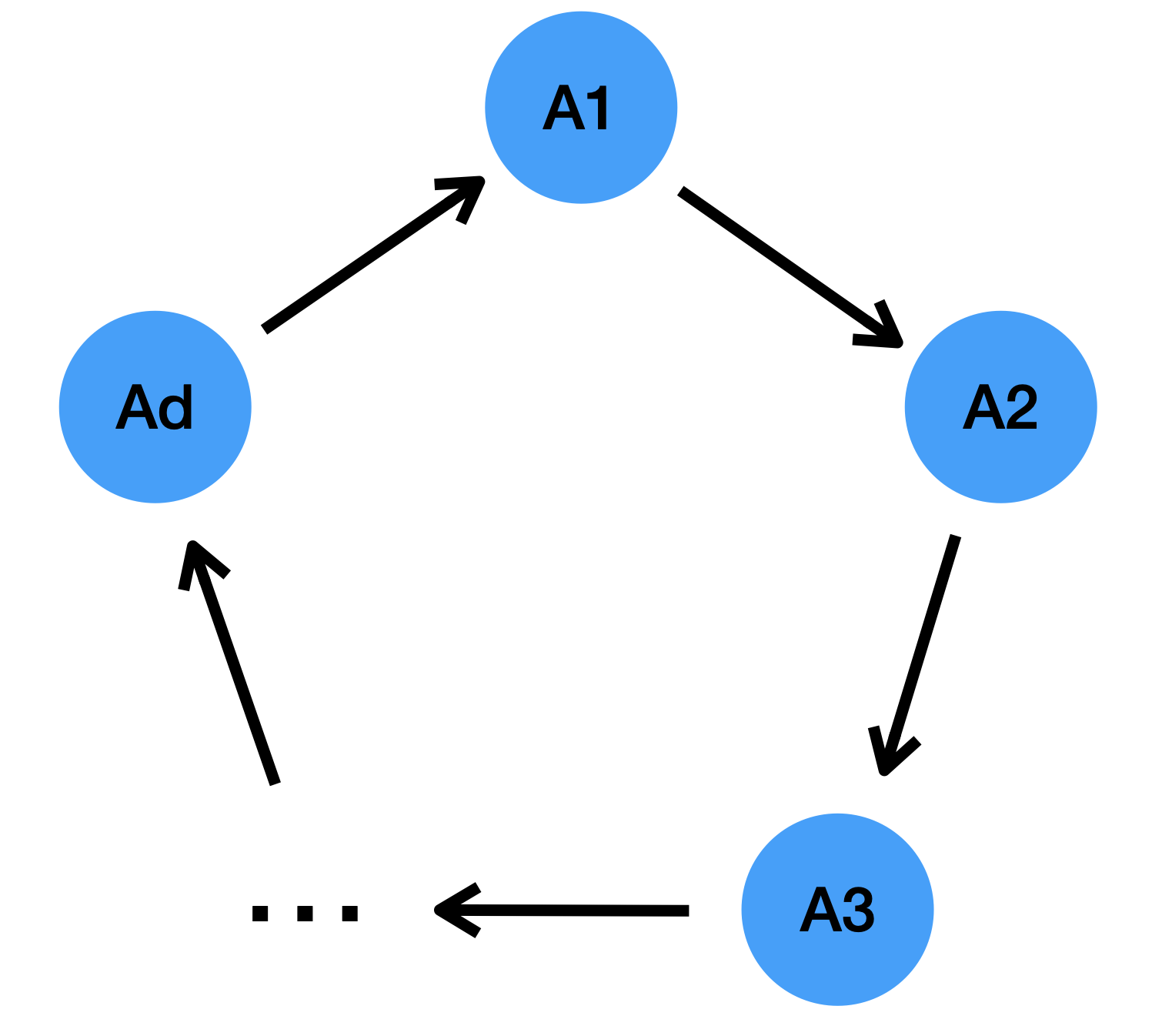}
\caption{Schematic representation of the \textit{cyclic feedback loop} \eqref{ODE dD}: The blue circles represent the molecules of the network, and the arrows between them regulation, which can be positive ($f_i$ increasing) or negative ($f_i$ decreasing).}
\label{figure}
\end{figure}




This model is a generalisation of a gene regulatory network initially proposed by Goodwin \cite{goodwin1963temporal, goodwin1965oscillatory}, usually referred to as \textit{Cyclic feedback loop}, which represents interactions between genes, mRNAs, enzymes and proteins called repressors which have the ability to inhibit the expression of some genes. In system  \eqref{ODE dD}, $x_1, ..., x_d$ represent the concentration of each of the molecules involved in the network (denoted $A_1$, ..., $A_n$), and $f_1,..., f_d$ the regulation between them. The system is assumed to be cyclic ($A_i$ regulates $A_{i+1}$ and only $A_{i+1}$, as illustrated by Figure \ref{figure}), and each regulation can be positive ($f_i$ increasing) or negative ($f_i$ decreasing). 
The relevance of these cyclic models has been established in \cite{angeli2004detection} and \cite{gardner2000construction} where some theoretical predictions (oscillatory phenomena and bistability) have been observed experimentally. This highlights the importance of understanding the dynamical behaviour of such systems \textit{i.e.} determining the number of stable equilibrium points and their basins of attraction, as well as the possible existence of periodic solutions or chaotic behaviours. 

System \eqref{ODE dD} has been, for some specific choices of $f_i$,  the subject of several theoretical studies \cite{banks1978stability, muller2006generalized, selgrade1980asymptotic, smith1987oscillations, smith1988systems, widder2007dynamic}. In these papers, restrictions on the functions $f_i$ were notably imposed by the necessity to compute the value of the equilibrium points of the system, which is intricate when more than two functions are not affine, and are not identical. In the present paper, we follow a method initiated by Cherry and Adler \cite{cherry2000make} allowing to avoid explicitly computing the equilibrium points. In the two dimensional case, which writes



\begin{align}
\begin{cases}
\dot x_1=\alpha_1f_1(x_2)-x_1\\
\dot x_2=\alpha_2f_2(x_1)-x_2
\end{cases}, 
\label{ODE 2D}
\end{align}the following results have already been established:

\begin{enumerate}
\item If $f_1$ and $f_2$ are both increasing or both decreasing, and if 
\begin{align}
\underset{x>0}{\sup}\left\lvert \frac{xf_1'(x)}{f_1(x)}\right\rvert \underset{x>0}{\sup}\left\lvert \frac{xf_2'(x)}{f_2(x)}\right\rvert>1,
\label{condition Cherry Adler}
\end{align}
then there exist values of $(\alpha_1, \alpha_2)\in {\R_+^*}^2$ such that system \eqref{ODE 2D} is multistable \textit{i.e.} there exist at least two equilibrium points which are asymptotically stable \cite{cherry2000make}. 
\item If $f_1$ and $f_2$ are both increasing or both decreasing, and if 
\begin{equation}
\frac{1}{\sqrt{\lvert f_1'\rvert}} \quad \text{and} \quad \frac{1}{\sqrt{\lvert f_2'\rvert}} \quad \text{are strictly convex \footnotemark }, 
\label{condtion GuiPouPouDu}
\end{equation}

\footnotetext{ It is in fact sufficient to assume that these two functions are convex, and that at least one of them is strictly convex. }

then system  \eqref{ODE 2D} is either monostable or bistable, \textit{i.e.} there exist exactly one or exactly two equilibrium points which are asymptotically stable. Moreover, it  is possible to determine, up to a set of measure zero, the set of parameters for $(\alpha_1, \alpha_2)$ for which the system is monostable and the set of parameters for which it is bistable \cite{guilberteau2021monostability}. 
\item All the solutions to system \eqref{ODE 2D} converge (even without assuming any monotonicity). 
\end{enumerate}
This last result is a direct application of the Poincaré-Bendixson theorem and the Dulac-Bendixson theorem \cite{perko2013differential}. 

A natural question at this stage is which of these properties generalise to higher dimensions ($d\geq 3$). A major result was achieved by Mallet-Paret and Smith in \cite{mallet1990poincare}, showing that the Poincaré-Bendixson theorem can be adapted to monotone feedback systems, including \eqref{ODE dD}.  Furthermore, a series of results of Hirsch \cite{hirsch1982systems, hirsch1984dynamical, hirsch1983differential} generalised and summarised in \cite{smith1988systems}, has shown that in the case where an even number of functions $f_i$ is decreasing, the solutions of \eqref{ODE dD} converge to an equilibrium point for almost every initial condition. It is well-known that this property does not hold when the number of decreasing functions is odd (in which case the system is often called `repressilator'), and that there can there exist stable orbits \cite{buse2010dynamical, muller2006generalized}. 

In this paper, we prove the following result, which is a generalisation of our previous paper on the two-dimensional case \cite{guilberteau2021monostability}: 



\begin{theorem}
Let us assume that $f_1,..., f_d\in \mathcal{C}^3(\R_+, \R_+^*)$ are monotonous, non-negative, and that (at least) one of them is bounded. Moreover, let us assume that the functions $\frac{1}{\sqrt{\lvert f'_1 \rvert }}, ..., \frac{1}{\sqrt{\lvert f'_d \rvert }}$ are defined and convex, and that (at least) one of them is strictly convex. Lastly, let us denote by $n$ the number of these functions which are decreasing, and let
\[D:=\prod\limits_{k=1}^d{\underset{x>0}{\sup}\left\lvert \frac{xf_k'(x)}{f_k(x)}\right\rvert}\in (0, +\infty] \]

\begin{enumerate}
\item If $n$ is even, then
\begin{enumerate}[(i)]
\item If $D<1$, then for any $\alpha\in \left(\R_+^*\right)^d$ system \eqref{ODE dD} has a unique equilibrium point which is globally asymptotically stable. 
\item If $D>1$, then there exists a non-empty set $A_{\textrm{bis}}\subset \left(\R_+^*\right)^d$  such that 
\begin{itemize}
\item If $\alpha \in A_{\textrm{bis}}$, then system \eqref{ODE dD} has exactly two asymptotically stable equilibria, and the union of their basins of attraction is a dense open subset of  $\R_+^d$, with a complement of Lebesgue measure zero. 
\item If $\alpha \in \overline{A_{\textrm{bis}}}^C$, then system \eqref{ODE dD} has a unique equilibrium point which is globally asymptotically stable. 
\end{itemize}
\item If $d\geq 5$ and  $D>\frac{1}{\cos \left( \frac{2 \pi}{d}\right)^d}$, then there exists a non-empty set $A_{\textrm{per}}\subset A_{\textrm{bis}}$ such that if $\alpha\in A_{\textrm{per}}$, then system \eqref{ODE dD} has periodic solutions.
\end{enumerate}
\item If $n$ is odd, then system \eqref{ODE dD} has a unique equilibrium point. Moreover, if $d\geq 3$, then 
\begin{enumerate}[(i)]
\item If $D<\frac{1}{\cos \left(\frac{\pi}{d}\right)^d}$, then for any  $\alpha\in \left(\R_+^*\right)^d$,  this equilibrium point is asymptotically stable, and all the solutions of \eqref{ODE dD} either converge to this point or to a periodic orbit. 
\item If $D>\frac{1}{\cos \left(\frac{\pi}{d}\right)^d}$, then there exists $A_{\textrm{unst}}\subset \left(\R_+^* \right)^d$ a non-empty set such that
\begin{itemize}
\item If $\alpha \in A_{\textrm{unst}}$, then this equilibrium point is asymptotically unstable, and there exists a finite number of periodic solutions, among which at least one is asymptotically stable. Moreover, the set of initial conditions for which the solution converges to a periodic solution is a dense open subset of $\R_+^d$, and its complement, which is the set of initial conditions for which the solution converges to the equilibrium point, has Lebesgue measure zero. 
\item If $\alpha \in \overline{A_{\textrm{unst}}}^C$, this equilibrium point is asymptotically stable, and all the solutions of \eqref{ODE dD} either converge to this point or to a periodic orbit. 
\end{itemize}

\end{enumerate}
\end{enumerate}

\label{Theorem intro}
\end{theorem}

In each of these cases, the sets $A_{\textrm{bis}}$, $A_{\textrm{per}}$, $A_{\textrm{unst}}$ can be explicitly expressed, as we will show in Sections \ref{Even number of decreasing functions} and \ref{Odd number of decreasing functions}. 
Moreover, note that any Hill function (even shifted), \textit{i.e.} function of the form $x \mapsto \frac{1+\lambda x^r}{1+x^r}$, with $\lambda \in \R_+ \backslash \{1\}$ and $r\geq 1$, as well as linear functions, satisfy the convexity hypothesis of this theorem, which means that this result encompasses the other theoretical studies mentioned above \cite{banks1978stability, muller2006generalized, selgrade1980asymptotic, smith1987oscillations, smith1988systems, widder2007dynamic}. 

It is worth noting that, when $n$ is even, the set of initial conditions for which the solution converges to a periodic orbit has Lebesgue measure zero, and can therefore hardly be reached numerically. Nevertheless, we highlight this result since, up to our knowledge, the question of the existence of such periodic solution remained open, as mentioned in \cite{smith1988systems}. Moreover, the proof of this result, which uses the stable manifold theorem, seems to us non-trivial and worthwhile. We also note that condition 1 \textit{(iii)} is sufficient, but perhaps not necessary for the existence of periodic solution: in particular, this question for  $d\in \{3,4\}$ remains open. Lastly, we do not know if periodic solutions do exist under the hypotheses of 2 \textit{(i)} and in the second point of 2 \textit{(ii)}, but our study does not rule out this possibility.

This article, entirely dedicated to the proof of Theorem \ref{Theorem intro}, is organised as follows: after characterising the equilibrium points of system \eqref{ODE dD} in Section \ref{Characterisation}, we prove the theorem when $n$ is even (Section \ref{Even number of decreasing functions}), before dealing with the case where $n$ is odd, which is simpler, in the last section. 

\section{Characterisation of fixed points}
\label{Characterisation}
Throughout this note, we assume that $f_1, ..., f_d \in \mathcal{C}^3(\R_+, \R_+^*)$ are non-negative, monotonous and that at least one of these functions is bounded.  Note that, according to the regularity of these functions, the Cauchy-Lipschitz theorem ensures the local existence and the uniqueness of the solution of this equation for any initial condition $x_0\in \R^d_+$. Furthermore, the positivity of the functions  guarantees that the solutions remain in $\R^d_+$. Lastly, using the fact that one of them is bounded, we easily prove the global existence of solutions and the existence of a compact attractor set,\textit{i.e.} the existence of a compact set $K\subset \R^d_+$ such that for any initial condition $x_0\in\R_+^d$, there exists $T\geq 0$ such that $x(t)\in K$ for all $t\geq T$. 

We now make the additional assumption that $f_1,..., f_d$ are $\gamma^{1/2}-$convex, and that at least one of them is strictly $\gamma^{1/2}-$convex,  \textit{i.e.} they satisfy the following definition:

\begin{definition}[(\textbf{$\mathbf{\gamma^{1/2}-}$convexity})]
Let $f \in \mathcal{C}^3\left(\R_+,  \R_+\right)$ a non-negative and monotonous function. We say that $f$ is (strictly) $\gamma^{1/2}-$ convex if $\lvert f'\rvert >0$ and $\frac{1}{\sqrt{\lvert f' \rvert }}$ is (strictly) convex. 
\end{definition}


Note that this definition can be related to the definition of the Schwartzian derivative of $f$ defined by $S(f)=\frac{f'''}{f'}-\frac{3}{2}\left(\frac{f''}{f'}\right)^2$ by noting that 
$ (\frac{1}{\sqrt{\lvert f' \rvert }})''=-\frac{1}{2}\frac{1}{\sqrt{\lvert f' \rvert }} S(f)$. 

We relate here the equilibrium points of system \eqref{ODE dD} to the fixed points of an auxiliary function  $\tilde{f}$. The $\gamma^{1/2}-$convexity of the functions $f_i$ ensures that the number of equilibrium points cannot exceed three, and provides a criterion which characterises this exact number of equilibria. A similar approach was used, (with the Schwartzian derivative) for a particular case of this system in \cite{muller2006generalized}. 

We start by recalling some key properties of the $\gamma^{1/2}-$convexity, which have been established in \cite{guilberteau2021monostability}. 

\begin{prop}
Let $f, g$ be two $\gamma^{1/2}-$convex functions, $c>0$.
\begin{enumerate}[(i)]
\item $f\circ g$ and $cf$ are  $\gamma^{1/2}-$convex.
 Moreover, if $f$ or $g$ is strictly $\gamma^{1/2}-$convex, then $f\circ g$ is strictly $\gamma^{1/2}-$convex. 
\end{enumerate}
Let us now assume that $f$ is strictly $\gamma^{1/2}-$convex. Then:
\begin{enumerate}[(i)]
\setcounter{enumi}{1}
\item $cf$ is strictly $\gamma^{1/2}-$convex.
\item $f$ has at most three fixed points. 
\item If all the fixed points $x$ of $f$ satisfy $f'(x)<1$, then $f$ has a unique fixed point.
\item If there exists a fixed point of $f$ (denoted $x$) such that $f'(x)>1$, then $f$ has exactly three fixed points and the other two fixed points (denoted $ y$, $ z$) satisfy $f'( y)<1$ and $f'(z)<1$. 
\end{enumerate}
\label{prop gamma-half}
\end{prop}

\begin{ex}
For any $r\geq 1$, $a,b,c,d\geq 0$ such that $ad-bc\neq 0$, the function  $x\mapsto \frac{ax^r+b}{cx^r+d}$ is $\gamma^{1/2}-$convex. Moreover, if $r>1$, then it is strictly $\gamma^{1/2}-$convex. In particular, affine functions and Hill functions are $\gamma^{1/2}-$convex.\footnote{In the Appendix of \cite{guilberteau2021monostability} we show that many other usual sigmoid functions are strictly $\gamma^{1/2}-$convex.} 


\end{ex}

By definition, the point $\bar x= (\bar x_1, ..., \bar x_d)\in \R_+^d$ is an equilibrium  point of \eqref{ODE dD} if and only if
\begin{align*}
\begin{cases}
\bar x_1=\alpha_1f_1(\bar x_d)\\
\bar x_2=\alpha_2f_2(\bar x_1)\\
\vdots\\
\bar x_d = \alpha_d f_d(\bar x_{d-1})
\end{cases}
\Longleftrightarrow \quad 
\begin{cases}
\bar x_1=\alpha_1f_1(\bar x_d)\\
\bar x_2=\alpha_2 f_2(\bar x_1)\\
\vdots\\
\bar x_{d-1}=\alpha_{d-1} f_{d-1}(\bar x_{d-2})\\
\bar x_d=\alpha_d f_d \circ \alpha_{d-1}f_{d-1}\circ... \circ \alpha_1 f_1(\bar x_d)\\
\end{cases}. 
\end{align*} 

Thus, the number of equilibrium points of $\eqref{ODE dD}$ is equal to the number of fixed points of $\tilde{f}:=\alpha_d f_d \circ ... \circ \alpha_1 f_1$. Since $\tilde{f}(0)>0$ and $\tilde{f}$ is bounded, it proves in particular that \eqref{ODE dD} has at least one equilibrium point. 
Note that, according to the first two properties of Proposition \ref{prop gamma-half}, $\tilde{f}$ is strictly $\gamma^{1/2}-$ convex, and a direct computation shows that 
\[\tilde{f}'(\bar x_d)=\prod\limits_{i=1}^d{\alpha_i f_i'(\bar x_{i-1})}. \]

We can thus apply the fourth and the fifth properties of Proposition \ref{prop gamma-half} to $\tilde{f}$ to derive the following lemma:

\begin{lem}
Let us denote, for all $\alpha, x \in \R_+^d$,  $p_{x}^\alpha=\prod\limits_{i=1}^d{\alpha_i f_i'( x_{i-1})}$. 
\begin{itemize}
\item If all the equilibrium points of system \eqref{ODE dD} satisfy $p_{\bar x}^\alpha<1$, then this system has a unique equilibrium point. 
\item If there exists an equilibrium point of system \eqref{ODE dD} (denoted $\bar x$) such that $p_{\bar x}^\alpha>1$, then this system has exactly three equilibrium points, and the other two points (denoted $\bar y,\bar z$) satisfy $p_{\bar y}^\alpha<1$, $p_{\bar z}^\alpha<1$. 
\end{itemize}
\label{lem characterisation fixed points}
\end{lem}
Hence, the value of $p^\alpha_{\bar x}$ characterises the number of fixed points the system has. In the following section, we show that it also determines the dimension of the basin of attraction of $\bar x$.  

\section{Even number of decreasing functions}
\label{Even number of decreasing functions}
In the case where $n$ is even,  one easily checks that system \eqref{ODE dD} is an \textit{irreducible type K monotone system} in the sense defined in \cite{smith1988systems}. 
As seen in the previous section, this system has a finite  number of equilibrium points (at most three), and a compact attractor set: thus, we can apply Theorem 2.5 and Theorem 2.6 of \cite{smith1988systems} which prove that the union of the basins of attraction of the equilibrium points is dense, and that the complement of this set has Lebesgue measure zero. 

 In this section, we complete this result in two ways:
\begin{itemize}
\item We determine, for given functions $f_1,..., f_d$, a set of parameters $A_{\textrm{bis}}$ such that system \eqref{ODE dD} is bistable (\textit{i.e.} has exactly two asymptotically stable equilibrium points) if $\alpha\in A_{\textrm{bis}}$, and monostable (\textit{i.e.} has exactly one asymptotically stable equilibrium point) if $\alpha \in \overline{A_{\textrm{bis}}}^C$. 
\item We determine a set $A_{\textrm{per}}\subset  A_{\textrm{bis}}$ such that system \eqref{ODE dD} has some periodic solutions if $\alpha \in A_{\textrm{per}}$.
\end{itemize}
Note that this last point does not mean that periodic solutions do not exist when $\alpha \notin A_{per}$, and that, in all cases, the set of initial conditions for which the solution converges to a periodic solution has Lebesgue measure zero (as a corollary of  \cite{smith1988systems}).

In order to prove these two points, we determine the dimension of the basin of attraction of an equilibrium point $\bar x$, as a function of $p_{\bar x }^\alpha$. For any equilibrium point $\bar x$, we denote its basin of attraction $B_{\bar x}.$ Our reasoning is based on the stable manifold theorem (the proof of which can be found for instance in \cite{perko2013differential}),  that we recall:

\begin{theorem*}[(Stable manifold)]
Let $F\in \mathcal{C}^1(\R^d, \R^d)$ a vector field, and let $\bar x \in \R^d$ such that $F(\bar x)=0_{\R^d}$. If $\bar x$ is a hyperbolic equilibrium point, \textit{i.e.} if all the eigenvalues of $\textrm{Jac}\, F(\bar x)$ have a non-zero real part, then the basin of attraction of $\bar x$ is a manifold of dimension $m$, where $m$ is the number of eigenvalues of $\textrm{Jac}\, F(\bar x)$ with a negative real part. 
\end{theorem*}
In order to apply this theorem, we need to compute the eigenvalues of the Jacobian matrix associated to system \eqref{ODE dD} which writes, at a given point $x\in \R^d$


\begin{align*}
M_x^\alpha=\begin{pmatrix}
-1 & 0 & \cdots & 0 &  \alpha_1 f_1'( x_d)\\
\alpha_2 f_2'(x_1) & -1 & 0 & \cdots &0\\
0 & \ddots & \ddots & \ddots & \vdots\\
\vdots & \ddots &\ddots &\ddots & 0\\ 
0 & \cdots & 0 &\alpha_d f_d'(x_{d-1}) & -1
\end{pmatrix}. 
\end{align*}

Thus, the characteristic polynomial of $M_{x}^\alpha$ is easily computed to be 

\[(-1)^d \left((\lambda+1)^d-\prod\limits_{i=1}^d{\alpha_i f_i'(x_{i-1})} \right)=(-1)^d \left((\lambda+1)^d-p_x^\alpha\right) \]
Since $n$ is even,  $p_x^\alpha>0$ and hence the spectrum of $M_x^\alpha$ is given by

\[\mathrm{Sp}\left(M_x^\alpha\right)=\left\{(p_x^\alpha)^{1/d}e^{2k \pi i/d}-1, k \in \{0,..., d-1\}\right\}. \]

We deduce that 
\begin{enumerate}[(i)]
\item If $p_x^\alpha<1$, then all the eigenvalues of $M_x^\alpha$ have a negative real part.
\item If $d\in \{3,4\}$ and $p_x^\alpha>1$, or if $d\geq 5$ and $p_x^\alpha\in \left(1, \frac{1}{\cos\left( \frac{2\pi}{d}\right)^d}\right)$, then $M_x^\alpha$ has exactly $d-1$ eigenvalues with a negative real part, and one with a positive real part. 
\item If $d\geq 5$, and if $p_x^\alpha>\frac{1}{\cos\left( \frac{2\pi}{d}\right)^d}$, then $M_x^\alpha$ has at most $d-3$ eigenvalues with a negative real part. Moreover, if $d\leq 8$, or $d\in \llbracket 4j+1,4j+4 \rrbracket$ ($j\in \mathbb{N}\backslash \{0,1\}$), and for any $k\in \{2,..., j\}$, $p_{x}^\alpha\neq \frac{1}{\cos\left( \frac{2\pi k}{d}\right)^d}$, then all eigenvalues of $M_x^\alpha$ have a non-zero real part (where for all $a,b\in \R$, $a<b$, $ \llbracket a,b \rrbracket=[a,b]\cap \mathbb{N}$.)
\end{enumerate}

Thus, the stable manifold theorem yields
\begin{lem}
Let us assume that $n$ is even, and let $\bar x\in \R^d$ be an equilibrium point of \eqref{ODE dD}.
\begin{enumerate}[(i)]
\item If $p_{\bar x}^\alpha<1$, then $\mathrm{dim}\left(B_{\bar x}\right)=d$, \textit{i.e.} $B_{\bar x}$ is an open set.
\item If $d\in \{3,4\}$ and $p_{\bar x}^\alpha>1$, or if $d\geq 5$ and $p_{\bar x}^\alpha\in \left(1, \frac{1}{\cos\left( \frac{2\pi}{d}\right)^d}\right)$, then $\mathrm{dim}\left(B_{\bar x}\right)=d-1$.
\item If $d\geq 5$,   $p_{\bar x}^\alpha>\frac{1}{\cos\left( \frac{2\pi}{d}\right)^d}$, and $p_{\bar x}^\alpha \notin S_d$, with 
\begin{align*}
S_d:=\begin{cases}
\quad \emptyset \quad &\textrm{if} \quad d\in  \{5,6,7,8\}\\
\left\{ \frac{1}{\cos\left( \frac{2\pi k}{d}\right)^d}, k\in \{2,... j\} \right\} \quad &\textrm{if}\quad  d\in \llbracket 4j+1,4j+4 \rrbracket, \quad  j\geq 2
\end{cases}, 
\end{align*}
then $\mathrm{dim}\left(B_{\bar x}\right)\leq d-3$.
\end{enumerate}
\label{lem stability}
\end{lem}




Moreover, we easily check that, in all cases, $M_{\bar x}^\alpha$ has an odd number of eigenvalues with a positive real part. Thus, since $n$ is even, we get from the main theorem of \cite{mallet1990poincare} that any solution converges to an equilibrium point or to a periodic orbit. 

We will now use Lemmas \ref{lem stability} and \ref{lem dimension unstable} to prove Theorem \ref{Theorem intro} in a more precise form which specifies the sets $A_{\mathrm{bis}}$ and $A_{\mathrm{per}}$, in the case where $n$ is even. Before stating it, we give a last lemma linking the dimension of the basin of attraction of the unstable equilibrium point to the existence of divergent solutions (which thus converge to periodic solutions), in the bistable case.

\begin{lem}
Let `$\dot x =F(x)$' be an  ODE which has exactly three equilibrium points, (denoted $\bar x$, $\bar y$, $\bar z$) assumed hyperbolic, and let us assume that for all initial condition $x_0\in \R^d$, the solution of this ODE is defined on $\R_+$ and converges. If $\bar y$, $\bar z$ are asymptotically stable, then $\mathrm{dim}\left(B_{\bar x}\right)=d-1.$
\label{lem dimension unstable}
\end{lem}

\begin{proof}
First, let us note that, according to the stable manifold theorem, $B_{\bar y}$ and  $B_{\bar z}$ are two open sets, and that $B_{\bar x}$ is a manifold. Since, by hypothesis, all the solutions converge, $\R_+^d=B_{\bar x}\cup B_{\bar y} \cup B_{\bar z}$, which means that $B_{\bar x}$ separates $\R_+^d$ in two open sets. As shown in \cite{hurewicz2015dimension} (Corollary 1 of Theorem IV 4), and by the connectedness of $\R_+^d$, this is possible only if $\mathrm{dim}\left(B_{\bar x}\right)=d-1$. 
\end{proof}

Before stating our theorem, let us introduce the functions $\Gamma$ and $G$, defined for any $x\in \R_+^d$ by $\Gamma(x)=\left( \frac{x_1}{f_1(x_d)}, \frac{x_2}{f_2(x_1)}, ..., \frac{x_d}{f_d(x_{d-1})} \right)$ and $G(x)=\prod\limits_{i=1}^d{\frac{x_if'(x_i)}{f(x_i)}}$, and the sets $E_{\mathrm{bis}}=\left\{ x\in \R_+^d: G(x) > 1 \right\}$ and if $d\geq 5$,  $E_{\textrm{per}}=\left\{ x\in \R_+^d: G(x)>\frac{1}{\cos \left(\frac{2 \pi }{d}\right)^d}, G(x)\notin S_{d}\right\}$. Note that $E_{\textrm{per}}\subset E_{\textrm{bis}}$. 




\begin{theorem}
Let us assume that $n$ is even.
\begin{enumerate}[(i)]
\item If $\alpha\in \overline{\Gamma(E_{\mathrm{bis}})}^C$, then \eqref{ODE dD} has a unique equilibrium point which is globally asymptotically stable. 
\item If $\alpha\in \Gamma(E_{\mathrm{bis}})$, then \eqref{ODE dD} has exactly three equilibrium points, among which two are asymptotically stable and one is asymptotically unstable. Moreover, the union of the basins of attraction of the two stable equilibria is a dense open subset of $\R_+^d$.
\item If $d\geq 5$, and if $\alpha\in \Gamma(E_{\mathrm{per}})$, then there exist periodic solutions of \eqref{ODE dD}. 
\end{enumerate}
\end{theorem}

\begin{proof}
First, let us note that, according to the definitions of $\Gamma$ and $G$, $\bar x\in \R_+^d$ is a fixed point of system \eqref{ODE dD} if and only if $\alpha=\Gamma(\bar x)$, and that for any equilibrium point $\bar x$ of \eqref{ODE dD}, 

\[p_{\bar x}^\alpha=p_{\bar x}^{\Gamma(\bar x)}=G(\bar x).\] 
\begin{enumerate}[(i)]
\item Let us assume that  $\alpha \in \overline{\Gamma(E_{\mathrm{bis}})}^C$, and let $\bar x$ be a fixed point of \eqref{ODE dD}. Since $\Gamma(\overline{E_{\mathrm{bis}}}) \subset \overline{\Gamma(E_{\mathrm{bis}})} $, and $\alpha=\Gamma(\bar x)$, $\bar x \in \overline{\Gamma(E_{\mathrm{bis}})}^C$, which means, by definition of $E_{\mathrm{bis}}$, that $p^\alpha_{\bar x}=G(\bar x)<1$. Since this equality holds for any equilibrium point, we conclude by Lemma \ref{lem characterisation fixed points}, that system \eqref{ODE dD} has a unique equilibrium point, which is asymptotically stable, by Lemma \ref{lem stability}. Since \eqref{ODE dD} is an irreducible type K monotone system, this unique equilibrium is in fact globally asymptotically stable \cite{smith1988systems}.
\item  Let us assume that  $\alpha \in \Gamma (E_{\textrm{bis}})$. Then, there exists $\bar x \in E_{\textrm{bis}}$ such that $\alpha=\Gamma(\bar x)$. The point $\bar x$ is thus an equilibrium point     of \eqref{ODE dD} which satisfies $p_{\bar x}^\alpha=G(\bar x)>1$. Therefore, Lemmas \ref{lem characterisation fixed points},  \ref{lem stability} and the result of \cite{smith1988systems} mentioned at the beginning of the section yield the result. 


\item Let us assume that  $\alpha \in \Gamma (E_{\textrm{per}})$. Since $E_{\textrm{per}}\subset E_{\textrm{bis}}$, \eqref{ODE dD} has exactly three equilibrium points, denoted $\bar x$, $\bar y$ and $\bar z$, with $\bar x \in E_{\textrm{per}}$, and $\bar y$ and $\bar z$ which are asymptotically stable. By definition of $E_{\textrm{per}}$, $\textrm{dim}\left(B_{\bar x}\right)\leq d-3$. By the contrapositive of Lemma \ref{lem dimension unstable}, system \eqref{ODE dD} has some divergent solutions, which thus converge to a periodic orbit, according to \cite{mallet1990poincare}. 
\end{enumerate}
\end{proof}

\begin{rem}
As mentioned above, this theorem is a more precise version of Theorem \ref{Theorem intro}: we find the statement of the latter by defining $A_{\textrm{bis}}=\Gamma\left(E_{\textrm{bis}} \right)$ and $A_{\textrm{per}}=\Gamma\left( E_{\textrm{per}} \right)$, and by noting that $E_{\textrm{bis}}$ (resp. $E_{\textrm{per}}$) is empty if and only if $D\leq 1$ (resp. $D\leq \frac{1}{\cos \left( \frac{2\pi}{d} \right)^d}$). 
\end{rem}

\section{Odd number of decreasing functions}
\label{Odd number of decreasing functions}
We now deal with the case where $n$ is odd. This case is simpler, since the system has a unique equilibrium point under this hypothesis. Nevertheless, we make  weaker conclusions regarding the global behaviour of the system, since it is not an irreducible type K monotone system (see \cite{smith1988systems}). Thus, we simply study the linearised system at the neighbourhood of the equilibrium point, and we conclude with \cite{mallet1990poincare}, which guarantees that the solutions either converge to this equilibrium point, or to a periodic orbit.

Let us denote
$E_{unst}:=\left\{ x\in \R^d: G(x)<-\frac{1}{\cos \left( \frac{\pi}{d} \right)^d} \right\}. $


We get the following result: 
\begin{theorem}
Let us assume that $n$ is odd, and that $d\geq 3$. Then, system \eqref{ODE dD} has a unique equilibrium point. Moreover, 
\begin{enumerate}[(i)]
\item If $\alpha \in \overline{\Gamma \left(E_{\textrm{unst}}\right)}^C$, then this equilibrium point is asymptotically stable. Moreover, all the solutions of \eqref{ODE dD} either converge to this point or to a periodic orbit. 
\item If $\alpha\in \Gamma(E_{\textrm{unst}})$, then this equilibrium point is asymptotically unstable, Moreover, the set of initial conditions for which the solution converges to a periodic solution is a dense open subset of $(\R_+^*)^d$, and its complement, which is the set of initial conditions for which the solution converges to the equilibrium point, has Lebesgue measure zero.  
\end{enumerate}
\end{theorem}

\begin{proof}
We use the same notations as in the previous section. Since $n$ is odd, $\tilde{f}$ is decreasing with $\tilde{f}(0)>0$, it has a unique fixed point, which implies that system \eqref{ODE dD} has a unique equilibrium point, that we denote $\bar x$. This point is asymptotically stable if all the eigenvalues of $M_{\bar x}^\alpha$ have a negative real part, and asymptotically unstable if at least one of these eigenvalues has a positive real part. Since $n$ is odd, 
\[\mathrm{Sp}\left(M_{\bar x}^\alpha\right)=\left\{(\lvert p_x^\alpha\rvert )^{1/d}e^{(2k+1)\pi i/d}-1, k \in \{0,..., d-1\}\right\},  \] 
which implies that $\bar x$ is stable if $\lvert p_{\bar x}^\alpha \rvert< \frac{1}{\cos \left( \frac{\pi}{d} \right)^d}$, and unstable if $\lvert p_{\bar x}^\alpha \rvert> \frac{1}{\cos \left( \frac{\pi}{d} \right)^d}$. We conclude by noting that $\lvert p_{\bar x}^{\alpha}\rvert =-p_{\bar x}^{\alpha}=- p_{\bar x}^{\Gamma(\bar x)} =-G(\bar x)$, and by applying the main theorem of \cite{mallet1990poincare} for the first point, and Theorem 4.3 of this same article for the second one. 

\end{proof}

We recover the result of Theorem \ref{Theorem intro} by defining $A_{\textrm{unst}}=\Gamma(E_{\textrm{unst}})$.

\section*{Acknowledgements}
The author thanks Nastassia Pouradier Duteil and Camille Pouchol for their proofreading and their guidance throughout the writing of this paper. 

\bibliographystyle{plain}
\bibliography{biblio_Note_ODE}

\begin{thebibliography}{10}

\bibitem{angeli2004detection}
David Angeli, James~E Ferrell~Jr, and Eduardo~D Sontag.
\newblock Detection of multistability, bifurcations, and hysteresis in a large
  class of biological positive-feedback systems.
\newblock {\em Proceedings of the National Academy of Sciences},
  101(7):1822--1827, 2004.

\bibitem{banks1978stability}
HT~Banks and JM~Mahaffy.
\newblock Stability of cyclic gene models for systems involving repression.
\newblock {\em Journal of Theoretical Biology}, 74(2):323--334, 1978.

\bibitem{buse2010dynamical}
Olguta Buse, Rodrigo P{\'e}rez, and Alexey Kuznetsov.
\newblock Dynamical properties of the repressilator model.
\newblock {\em Physical Review E}, 81(6):066206, 2010.

\bibitem{cherry2000make}
Joshua~L Cherry and Frederick~R Adler.
\newblock How to make a biological switch.
\newblock {\em Journal of theoretical biology}, 203(2):117--133, 2000.

\bibitem{gardner2000construction}
Timothy~S Gardner, Charles~R Cantor, and James~J Collins.
\newblock Construction of a genetic toggle switch in escherichia coli.
\newblock {\em Nature}, 403(6767):339--342, 2000.

\bibitem{goodwin1965oscillatory}
Brian~C Goodwin.
\newblock Oscillatory behavior in enzymatic control processes.
\newblock {\em Advances in enzyme regulation}, 3:425--437, 1965.

\bibitem{goodwin1963temporal}
Brian~C Goodwin et~al.
\newblock Temporal organization in cells. a dynamic theory of cellular control
  processes.
\newblock {\em Temporal organization in cells. A dynamic theory of cellular
  control processes.}, 1963.

\bibitem{guilberteau2021monostability}
Jules Guilberteau, Camille Pouchol, and Nastassia Pouradier~Duteil.
\newblock Monostability and bistability of biological switches.
\newblock {\em Journal of Mathematical Biology}, 83(6-7):65, 2021.

\bibitem{hirsch1982systems}
Morris~W Hirsch.
\newblock Systems of differential equations which are competitive or
  cooperative: I. limit sets.
\newblock {\em SIAM Journal on Mathematical Analysis}, 13(2):167--179, 1982.

\bibitem{hirsch1983differential}
Morris~W Hirsch.
\newblock Differential equations and convergence almost everywhere in strongly
  monotone semiflows.
\newblock {\em Contemp. Math}, 17:267--285, 1983.

\bibitem{hirsch1984dynamical}
Morris~W Hirsch.
\newblock The dynamical systems approach to differential equations.
\newblock {\em Bulletin of the American mathematical society}, 11(1):1--64,
  1984.

\bibitem{hurewicz2015dimension}
Witold Hurewicz and Henry Wallman.
\newblock {\em Dimension Theory (PMS-4), Volume 4}, volume~63.
\newblock Princeton university press, 2015.

\bibitem{mallet1990poincare}
John Mallet-Paret and Hal Smith.
\newblock The poincar{\'e}-bendixson theorem for monotone cyclic feedback
  systems.
\newblock {\em Journal of Dynamics and Differential Equations}, 2(4):367--421,
  1990.

\bibitem{muller2006generalized}
Stefan M{\"u}ller, Josef Hofbauer, Lukas Endler, Christoph Flamm, Stefanie
  Widder, and Peter Schuster.
\newblock A generalized model of the repressilator.
\newblock {\em Journal of mathematical biology}, 53:905--937, 2006.

\bibitem{perko2013differential}
Lawrence Perko.
\newblock {\em Differential equations and dynamical systems}, volume~7.
\newblock Springer Science \& Business Media, 2013.

\bibitem{selgrade1980asymptotic}
James~F Selgrade.
\newblock Asymptotic behavior of solutions to single loop positive feedback
  systems.
\newblock {\em Journal of Differential Equations}, 38(1):80--103, 1980.

\bibitem{smith1987oscillations}
Hal Smith.
\newblock Oscillations and multiple steady states in a cyclic gene model with
  repression.
\newblock {\em Journal of mathematical biology}, 25(2):169--190, 1987.

\bibitem{smith1988systems}
Hal~L Smith.
\newblock Systems of ordinary differential equations which generate an order
  preserving flow. a survey of results.
\newblock {\em SIAM review}, 30(1):87--113, 1988.

\bibitem{widder2007dynamic}
Stefanie Widder, Josef Schicho, and Peter Schuster.
\newblock Dynamic patterns of gene regulation i: simple two-gene systems.
\newblock {\em Journal of theoretical biology}, 246(3):395--419, 2007.

\end{thebibliography}

\end{document}